\documentclass{article}

\usepackage{arxiv}

\usepackage{amsmath}
\usepackage{graphicx}
\usepackage{amsthm}
\usepackage{amssymb}
\usepackage{verbatim}
\usepackage{latexsym}
\usepackage{epstopdf}
\usepackage{subfigure}
\usepackage{bbm}
\usepackage{float}

\theoremstyle{thmstyletwo}%
\newtheorem{theorem}{Theorem}
\numberwithin{equation}{section}

\newcommand{\EE}{\mathbb{E}}

\newcommand{\RR}{\mathbb{R}}

\newtheorem{lemma}[theorem]{Lemma}
\newtheorem{coro}[theorem]{Corollary}

\newtheorem{assp}[theorem]{Assumption}

\newtheorem{expl}[theorem]{Example}

\newtheorem{rmk}[theorem]{Remark}

\def\a{\alpha}\def\o{\theta} \def\d{\delta}
\def\e{\varepsilon}   \def\s{\sigma}

\def\D{\Delta} \def\k{\kappa}

\title{The semi-implicit Euler-Maruyama method for nonlinear non-autonomous stochastic differential equations driven by a class of L\'evy processes}

\author{
 Xiaotong Li, Wei Liu, Hongjiong Tian \\
 Department of Mathematics, Shanghai Normal University, Shanghai, 200234, China\\
 \texttt{x.t.li@foxmail.com; weiliu@shnu.edu.cn; hjtian@shnu.edu.cn}
}

\begin{document}
\maketitle

\begin{abstract}
The strong convergence of the semi-implicit Euler-Maruyama (EM) method for stochastic differential equations with non-linear coefficients driven by a class of L\'evy processes is investigated. The dependence of the convergence order of the numerical scheme on the parameters of the class of L\'evy processes is discovered, which is different from existing results. In addition, the existence and uniqueness of numerical invariant measure of the semi-implicit EM method is studied and its convergence to the underlying invariant measure is also proved. Numerical examples are provided to confirm our theoretical results.
\end{abstract}

\keywords{L\'evy process \and stochastic differential equation \and semi-implicit Euler-Maruyama method \and strong convergence \and invariant measure}

\section{Introduction}\label{Intro}
L\'evy processes and stochastic differential equations (SDEs) driven by L\'evy processes have been widely employed to model uncertain phenomena in various areas \cite{AkgirayBooth1988,BardhanChao1993,ContTankov2003,ContVoltchkova2004,Mandelbrot1997}.
Since different classes of L\'evy processes may have distinct properties \cite{D.Applebaum2009,Sato2013}, much attention should be paid for characteristics of different types of  L\'evy processes,  when  analyzing numerical methods for SDEs driven by L\'evy processes is done.
\par
Numerical methods for SDEs driven different classes of L\'evy processes were proposed and studied. Higham and Kloeden \cite{HighamKloeden2007} studied the semi-implicit Euler-Maruyama (EM) method for solving SDEs with Poisson-driven jumps and obtained its optimal strong convergence order. More recently, the tamed Euler approach \cite{DareiotisKumarSabanis2016}, the tamed Milstein method \cite{KumarSabanis2017}, general one-step methods \cite{ChenGanWang2019}, and  the truncated EM method \cite{Zhang2021} were proposed for solving SDEs driven by L\'evy processes with super-linearly growing coefficients. All the works mentioned above focused on handling the super-linearity of the coefficients but not cared about the effects of the types of  L\'evy processes on the convergence order.
\par
On the contrary, different types of stable processes do lead to different convergence orders. Oliver and Dai \cite{OliverDai2017} investigated the strong convergence order of the EM method for SDEs with H\"older continuous coefficients in both time and spatial variables and driven by a truncated symmetric $\a$-stable process. Mikulevicius and Xu \cite{MikuleviciusXu2018} considered the convergence order of the EM approximation for SDEs driven by an $\a$-stable process  with a Lipschitz continuous coefficient and a H\"older drift. Huang and Liao \cite{HuangLiao2018} and  Huang and Yang \cite{HuangYang2021} extended  the EM method to stochastic functional differential equations  and distribution-dependent SDEs with H\"older drift and $\a$-stable process, respectively.  K\"uhn and Schilling \cite{KuhnSchilling2019} established strong convergence of the  EM method for a class of L\'evy process-driven SDEs. Although those papers \cite{OliverDai2017,MikuleviciusXu2018,HuangLiao2018,HuangYang2021,KuhnSchilling2019} discovered the dependence of the convergence orders on types of stable processes, they focused on the classical Euler type method, which ruled out the cases of super-linear coefficients \cite{HutzenthalerJentzenKloeden2011}.
\par
To our best knowledge, few existing works focused on both the super-linearity of the coefficients and the effects of the parameters of the class of L\'evy processes on the convergence order.
\par
In this paper, we focus on the class of L\'evy processes with the triplet $(0,0,\nu)$ satisfying $\int_{|z|< 1}|z|^{\gamma_0} \nu(dz)<\infty$ with $\gamma_0\in[1,2]$  and $\int_{|z|\geq 1}|z|^{\gamma_{\infty}} \nu(dz)<\infty$ with $\gamma_{\infty}>1$ (see Section \ref{Pre} for more details). And we derive the strong convergence order of the semi-implicit EM method that is dependent on $\gamma_0$. Such setting on the triplet of the L\'evy processes covers many interesting stable processes, such as the Lamperti stable process \cite{CPP2010} and the tempered stable process \cite{KT2013}. It is worth noting that the class of L\'evy processes we are referring to here does not include $\a$-stable, and there are moments of any order. The fundamental reason why this class of L\'evy process does not include $\a$-stable is that it does not have second moments, and in this case, only small pth moments ($p\in(0,\a)$) can be considered. Compared with existing results \cite{KuhnSchilling2019,OliverDai2017}, the SDEs studied in this work are allowed to have super-linear coefficients. Therefore, the EM method used in \cite{KuhnSchilling2019,OliverDai2017} may fail to convergence, due to the combined effects of the unbounded noise and super-linear terms in the coefficients \cite{HutzenthalerJentzenKloeden2011}.
\par
The invariant measure of SDEs with different classes of L\'evy processes is also essential for many stochastic models \cite{BYY2016,TJZ2018,Wang2013}. However, the explicit expression of the invariant measure is rarely obtained. Therefore, deriving the numerical invariant measure that is convergent to the underlying one is crucial in the applications of those models. In this paper, the existence and uniqueness of the numerical invariant measure of the semi-implicit EM method is proved and the convergence of the numerical invariant measure to the underlying one is verified. We should mention that there are many existing works on the numerical invariant measure of SDEs driven by  Brownian motion \cite{BaoShaoYuan2016,JiangWengLiu2020,LiuMao2015,LiMaoYin2019,Talay2002}, but few works are devoted to the case of the L\'evy processes.

\par
The main contributions of this paper are as follows.
\begin{itemize}
\item The strong convergence in the finite time of the semi-implicit EM method for SDEs driven by a class of L\'evy process is established and the dependence of the convergence order on the H\"older continuity and the L\'evy process is revealed.
\item The existence and uniqueness of the numerical invariant measure of the semi-implicit EM method is proved. Moreover, the convergence of the numerical invariant measure to the underlying one is obtained.
\end{itemize}
\par
This paper is organized as follows. In Section \ref{Pre},  some preliminaries are briefly introduced. In Section \ref{Main1}, we establish the finite-time strong convergence of the semi-implicit EM method.  The convergence of the  numerical invariant measure  is studied in Section \ref{Main2}. Numerical examples are  given in Section \ref{Numex} to illustrate our theoretical results and some concluding remarks are included in Section \ref{Con}.
\section{Preliminaries}\label{Pre}
Let $(\Omega, \mathcal{F},\mathbb{P})$ be a complete probability space with a filtration $\{\mathcal{F}_t\}_{t\geq 0}$ satisfying the usual conditions (i.e.,  it is right continuous and increasing in $t$ while $\mathcal{F}_0$ contains all $\mathbb{P}$-null sets). Let $|\cdot |$ denote the Euclidean norm in $\RR^d$. Let $B(t)$ be the $m$-dimensional Brownian motion defined on the probability space.
\par
A stochastic process $L=\{L(t), t\geq 0\}$ is called a $m$-dimensional L\'evy process if
\begin{itemize}
  \item $L(0)=0$ (a.s.);
  \item For any integer $n\geq 1$ and $0\leq t_0<t_1<\cdots <t_n\leq T$, the random variables $L(t_0), L(t_1)-L(t_0), \cdots , L(t_n)-L(t_{n-1})$  have independent and stationary increments;
  \item $L$ is stochastically continuous, i.e. for all $\epsilon>0$ and all $s\geq 0$,
  \begin{equation*}
    \lim_{t\rightarrow s}\mathbb{P}(|L(t)-L(s)|>\epsilon)=0.
  \end{equation*}
\end{itemize}
\par
The L\'evy-Khintchine formula \cite{D.Applebaum2009} for  the L\'evy process $L$ is
\begin{align*}
\varphi_L(\o)(t):=& \EE[e^{i\o \cdot L(t)}]\\
=& \exp\left(t(ib\cdot\o-\frac{1}{2}\o\cdot A\o+\int_{\RR^d \setminus\{0\}}(e^{i\o\cdot z}-1-i\o \cdot x\mathrm{1}_{|z|<1})\nu(dz))\right),
\end{align*}
where $b\in\RR^d$, $A\in \mathbb{R}^{d\times d}$ is a positive definite symmetric  matrix and the L\'evy measure $\nu$ is a $\s$-finite measure such that $\int_{\RR^d\setminus \{0\}}\min(1,x^2)\nu(dx)<\infty$. The L\'evy triplet $(b, A, \nu)$ is the characteristics of  the infinitely divisible random variable $L(1)$, which has a one-to-one correspondence with the L\'evy-Khintchine formula.



In this paper, we consider a SDE driven by the multiplicative Brownian motion and the additive L\'evy process of the form
\begin{equation}\label{sde}
y(t)=y(0)+\int_{0}^{t}f(s,y(s))ds+\int_{0}^{t}g(s,y(s))dB(s)+\int_{0}^{t}dL(s),\ \ t\in (0, T],
\end{equation}
where $f: \RR_{+}\times \RR^d\rightarrow \RR^d$, $g:\RR_{+}\times \RR^d\rightarrow\RR^{d\times m}$ and $\mathbb{E}|y(0)|^p<+\infty$ for $2\leq p<+\infty$.  For simplicity,  the L\'evy triplet is assumed to be  $(0,0,\nu)$ and  satisfies
 $\int_{|z|< 1}|z|^{\gamma_0} \nu(dz)<\infty$ with $\gamma_0\in[1,2]$  and $\int_{|z|\geq 1}|z|^{\gamma_{\infty}} \nu(dz)<\infty$ with $\gamma_{\infty}>1$ (see \cite{KuhnSchilling2019} for more details). In addition, we need some hypothesises on the drift and diffusion coefficients.
\begin{assp}\label{superlinear}
There exist constants $H>0$ and $\s>0$ such that
\begin{equation*}
| f(t,x)-f(t,y) |^2 \leq H(1+| x |^{\s}+| y |^{\s})| x-y |^2
\end{equation*}
for all $x,y\in\RR^d$ and $t\in[0,T]$.
\end{assp}
It can be observed from Assumption \ref{superlinear} that for all $t\in[0,T]$ and $x\in\RR^d$
\begin{equation}\label{superx}
|f(t,x)|^2\leq \widetilde{H}(1+|x|^{\s+2}),
\end{equation}
where $\widetilde{H}$ depends on $H$ and $\sup_{0\leq t\leq T}|f(t,0)|^2$.

\begin{assp}\label{Khasminskii}
There are constants $q\geq 2\s+2$ and $M>0$ such that
\begin{equation*}
x^{\mathrm{T}}f(t,x)+\frac{q-1}{2}| g(t,x) |^2 \leq M(1+| x |^2)
\end{equation*}
for all $x\in \RR^d$ and $t\in [0,T]$, where $\s$ is given in Assumption \ref{superlinear}.
\end{assp}
\begin{assp}\label{timeHolder}
There exist constants $K_1>0$, $K_2>0$, $\gamma_1\in(0,1)$ and $\gamma_2\in(0,1)$ such that
\begin{equation*}
| f(s,x)-f(t,x) | \leq K_1(1+| x |^{\s+1})| t-s |^{\gamma_1}
\end{equation*}
and
\begin{equation*}
| g(s,x)-g(t,x) | \leq K_2(1+| x |^{\s+1})| t-s |^{\gamma_2}
\end{equation*}
for all $x\in \RR^d$ and $s,t\in [0,T]$, where $\s$ is given in Assumption \ref{superlinear}.
\end{assp}
\begin{assp}\label{sideLip}
There exists a constant $K_3<-\frac{1}{2}$ such that
\begin{equation*}
(x-y)^{\mathrm{T}}(f(t,x)-f(t,y))\leq K_3| x-y |^2
\end{equation*}
for all $x,y\in \RR^d$ and $t\in[0,T]$.
\end{assp}
Setting $y=0$, we have the following inequality
\begin{equation}\label{sideLipx}
x^{\mathrm{T}}f(t,x)\leq M_1|x|^2+m_1,
\end{equation}
where $M_1=\frac{1}{2}+K_3<0$ and $m_1=\frac{1}{2}|f(t,0)|^2$.

\begin{assp}\label{lineargro}
There exists a constant $K_4>0$ with $K_4+2K_3<-1$ such that
\begin{equation*}
| g(t,x)-g(t,y)|^2 \leq K_4| x-y |^2
\end{equation*}
for all $x,y\in \RR^d$ and $t\in[0,T]$.
\end{assp}
It follows from Assumption \ref{lineargro} that
\begin{equation}\label{lineg}
|g(t,x)|^2\leq M_2|x|^2+m_2,
\end{equation}
where $M_2=K_4$, $m_2=|g(t,0)|^2$.
\par
The existence and uniqueness of the solution $y(t)$ of SDE \eqref{sde} is guaranteed \cite[Section 6.2]{D.Applebaum2009} under Assumptions \ref{sideLip} and \ref{lineargro}.

\section{Main results on the finite time strong convergence}\label{Main1}
Given $\D t\in(0,1)$ and time $T$, let  $\mathbf{N}=\lfloor T/\D t\rfloor$ and $t_i=i\D t, i=0,1,\ldots, \mathbf{N}$. The semi-implicit EM method of SDE \eqref{sde} is defined by
\begin{equation}\label{Numsde}
Y_{i+1}=Y_i + f(t_{i+1},Y_{i+1})\D t + g(t_i,Y_i)\D B_{i+1}+\D L_{i+1}, \quad i=0,1,2,\cdots \mathbf{N},
\end{equation}
where $ Y_i$ approximates $y(t_i)$, $\D B_{i+1}=B(t_{i+1})-B(t_i)$, and $\D L_{i+1}=L(t_{i+1})-L(t_i)$.
\par
The continuous-time semi-implicit EM solution of SDE \eqref{sde} is constructed in the following manner
\begin{equation}\label{consol}
Y(t)=Y_i, \quad t\in[t_i,t_{i+1}),\quad i=0,1,2,\cdots, \mathbf{N}.
\end{equation}
Our first main result is on the finite time strong convergence.
\begin{theorem}\label{Conver}
Suppose that Assumptions \ref{superlinear}--\ref{lineargro} hold. Then  the semi-implicit EM method \eqref{Numsde}-\eqref{consol}  for solving \eqref{sde} is convergent and satisfies
\begin{equation}\label{conver}
\EE|y(t)-Y(t)|^2\leq C\left(\D t^{2\gamma_1}+\D t^{2\gamma_2}+\D t\right),\ \ \forall t\in[0,T],
\end{equation}
where the positive constant $C$ is independent of $\D t$.
\end{theorem}
The proof of Theorem \ref{Conver} will be given in Section \ref{Section3.2}.
\par
Now we look at a special interesting case that  $g(\cdot,\cdot) \equiv 0$, i.e. the Brownian motion is removed
\begin{equation}\label{SDE2}
u(t)=u(0)+\int_{0}^{t}f(s,u(s))ds+\int_{0}^{t}dL(s).
\end{equation}
Then the semi-implicit EM method becomes
\begin{equation}
\label{1}
U_{i+1}=U_i + f(t_{i+1},U_{i+1})\D t +\D L_{i+1}
\end{equation}
and its continuous-time version is
\begin{equation}
\label{2}
U(t)=U_i, \quad t\in[t_i,t_{i+1}),\quad i=0,1,2,\cdots, \mathbf{N}.
\end{equation}
The second main result of this paper, which we think is more interesting, states that the convergence order depends on the parameter $\gamma_0$  of the L\'evy process and the H\"older's index $\gamma_1$.
\begin{coro}\label{Conver2}
Suppose that Assumptions \ref{superlinear}--\ref{sideLip} hold. Then the semi-implicit EM method \eqref{1}-\eqref{2} for solving \eqref{SDE2} is convergent  and satisfies
\begin{equation}\label{conver2}
\EE|u(t)-U(t)|^2\leq C\left(\D t^{2\gamma_1}+\D t^{\frac{2}{\gamma_0}}\right),\ \ \forall t\in[0,T],
\end{equation}
where the positive constant $C$ is independent of $\D t$.
\end{coro}
When the Brownian motion term is removed (i.e., $g(\cdot,\cdot)\equiv 0$), 
the main modification of the proof is that the exponent $\eta$ in Lemma \ref{continue} will be taken as $ \eta= \frac{2}{\gamma_0}$ in the proof of Theorem \ref{Conver}.

\begin{rmk}
The above conclusions prove that the convergence order under non--global Lipschitz conditions is related to the H\"older continuity index and the parameters of the calss of L\'evy processes. It is worth noting that when the noise is a specific stable process, including Lamperti stable, tempered stable, isotropic stable, and relativistic stable, the conclusions are applicable.
\end{rmk}


\subsection{Some necessary lemmas}

%
The existence and uniqueness of global solution has been given by \cite{ChenGanWang2019,GyongyKrylov1980}. In this section, we show the boundedness of the $p$th moment and the continuity of the solution.
\begin{lemma}\label{exactbound}
Suppose that Assumption \ref{superlinear} and \ref{Khasminskii} hold. Then, for  fixed $p\in[2,q)$,
\begin{equation*}
\EE | y(t) |^p \leq C_p, \ \ \forall t\in[0,T],
\end{equation*}
where $C_p:=C\left(p,T,M,\EE|y(0)|^p,\int_{|z|<1}|z|^2 \nu(dz),\int_{|z|\geq 1}|z| \nu(dz)\right)$.
\end{lemma}
\begin{proof}
Let $N$ be the Poisson measure on $[0,\infty)\times(\RR \backslash \{0\})$  with  $\EE N(dt,dz)=dt\nu (dz)$. Define the compensated Poisson random measure $\widetilde{N}(dt,dz)=N(dt,dz)-\nu(dz)dt.$ From the It\^o formula \cite{D.Applebaum2009}, it follows
\allowdisplaybreaks[4]
\begin{align*}
| y(t)|^p &= | y(0)|^p + \int_{0}^{t}p| y(s)|^{p-2} y^{\mathrm{T}}(s)f(s,y(s))ds+\int_{0}^{t}\frac{1}{2}p(p-1)| y(s)|^{p-2}| g(s,y(s))|^2 ds\\
&\quad + \int_{0}^{t}p| y(s)|^{p-2}y^{\mathrm{T}}(s)g(s,y(s))dB(s)+\int_{0}^{t}\int_{|z|<1}(|y(s)+z|^p-|y(s)|^p)\widetilde{N}(dz,ds)\\
&\quad + \int_{0}^{t}\int_{|z|\geq 1}(|y(s)+z|^p-|y(s)|^p)N(dz,ds)\\
&\quad + \int_{0}^{t}\int_{\RR^d}(|y(s)+z|^p-|y(s)|^p-p|y(s)|^{p-2}y^{\mathrm{T}}(s)z\mathbbm{1}_{(0,1)}(|z|))\nu(dz)ds\\
&= | y(0)|^p + \int_{0}^{t}p| y(s)|^{p-2}y^{\mathrm{T}}(s)f(s,y(s))+\frac{1}{2}p(p-1)| y(s)|^{p-2}| g(s,y(s))|^2 ds\\
&\quad + \int_{0}^{t}p| y(s)|^{p-2}y^{\mathrm{T}}(s)g(s,y(s))dB(s)+\int_{0}^{t}\int_{|z|<1}(|y(s)+z|^p-|y(s)|^p)\widetilde{N}(dz,ds)\\
&\quad + \int_{0}^{t}\int_{|z|\geq 1}(|y(s)+z|^p-|y(s)|^p)\widetilde{N}(dz,ds)\\
&\quad + 2\int_{0}^{t}\int_{|z|\geq 1}(|y(s)+z|^p-|y(s)|^p)\nu(dz)ds\\
&\quad + \int_{0}^{t}\int_{|z|<1}(|y(s)+z|^p-|y(s)|^p-p|y(s)|^{p-2}y^{\mathrm{T}}(s)z)\nu(dz)ds.
\end{align*}
Taking expectations on both sides, we obtain
\begin{align*}
\EE|y(t)|^p&\leq \EE|y(0)|^p+p\EE\int_{0}^{t}|y(s)|^{p-2}[y^{\mathrm{T}}(s)f(s,y(s))+\frac{p-1}{2}|g(s,y(s))|^2]ds\\
&\quad + 2\EE\int_{0}^{t}\int_{|z|\geq 1}(|y(s)+z|^p-|y(s)|^p)\nu(dz)ds\\
&\quad + \EE\int_{0}^{t}\int_{|z|<1}(|y(s)+z|^p-|y(s)|^p-p|y(s)|^{p-2}y^{\mathrm{T}}(s)z)\nu(dz)ds\\
&=: \EE|y(0)|^p+I_1+I_2+I_3.
\end{align*}
It follows from Assumption \ref{Khasminskii} that
\begin{align*}
I_1&= p\EE\int_{0}^{t}|y(s)|^{p-2}[y^{\mathrm{T}}(s)f(s,y(s))+\frac{p-1}{2}|g(s,y(s))|^2]ds\\
&\leq pM\EE\int_{0}^{t}|y(s)|^{p-2}(1+|y(s)|^2)ds\\
&\leq \EE [pMT] + 2pM\EE\int_{0}^{t}|y(s)|^pds.
\end{align*}
It is obvious that
\begin{align*}
I_2&= 2\EE\int_{0}^{t}\int_{|z|\geq 1}(|y(s)+z|^p-|y(s)|^p) \nu(dz)ds\\
&\leq 2p\EE\int_{0}^{t}\int_{|z|\geq 1}|y(s)|^{p-2}|y^{\mathrm{T}}(s)z| \nu(dz)ds\\
&\leq 2pT\EE\int_{|z|\geq 1}|z| \nu(dz) + 2p\EE\int_{0}^{t}|y(s)|^p\int_{|z|\geq 1}|z| \nu(dz)ds\\
&\leq C_p + C_p\EE\int_{0}^{t}|y(s)|^p ds.
\end{align*}
Note that for any $y_1, y_2\in\RR^d$
\begin{align*}
|y_1+y_2|^p-|y_1|^p-p|y_1|^{p-2}y_1^{\mathrm{T}}y_2&\leq C_p\int_{0}^{1}|y_1+\o y_2|^{p-2}|y_2|^2 d\o \\
&\leq C_p\left(|y_1|^{p-2}|y_2|^2+|y_2|^p\right).
\end{align*}
So,
\begin{align*}
I_3&= \EE\int_{0}^{t}\int_{|z|<1}\left[|y(s)+z|^p-|y(s)|^p-p|y(s)|^{p-2}y^{\mathrm{T}}(s)z\right]\nu(dz)ds\\
&\leq C_p\EE\int_{0}^{t}\int_{|z|<1}\left(|y(s)|^{p-2}|z|^2+|y(s)|^p\right)\nu(dz)ds\\
&\leq C_p\EE\int_{0}^{t}\int_{|z|<1}\left[|z|^2+(1+|z|^2)|y(s)|^p\right]ds\\
&\leq C_p+C_p\EE\int_{0}^{t}|y(s)|^p ds.
\end{align*}
Combining the estimates of $I_1$, $I_2$ and $I_3$, we arrive at
\begin{equation*}
\EE|y(t)|^p\leq C_p+C_p\EE\int_{0}^{t}|y(s)|^p ds.
\end{equation*}
The desired result follows from the Gronwall inequality.
\end{proof}

\begin{lemma}\label{continue}
Suppose that Assumptions \ref{superlinear} and \ref{lineargro} hold. Then, for any $1\leq \k <p$, $0\leq s < t\leq T$ and $|t-s|<1$, we have
\begin{equation}\label{yts}
\EE|y(t)-y(s)|^{\k}\leq C|t-s|^{\eta},
\end{equation}
where the constant $C$ only depends on $\widetilde{H}$, $C_p$ and $\k$ and $\eta=
\begin{cases}
\frac{\k}{2}, & g(\cdot,\cdot)\neq 0,\\
\frac{\k}{\gamma_0}, & g(\cdot,\cdot)\equiv 0.
\end{cases}$
\end{lemma}
\begin{proof}
For any $0\leq s<t\leq T$, the solution $y$ satisfies
\begin{equation*}
y(t)-y(s) = \int_{s}^{t}f(r,y(r))dr+\int_{s}^{t}g(r,y(r))dB(r)+(L(t)-L(s)).
\end{equation*}
By the H\"older inequality, the Burkholder--Davis--Gundy (BDG) inequality \cite{Mao2008}, the fractional moment estimate for the L\'evy process \cite{Kuhn2017}, and  the fact that
\begin{equation*}
L_t-L_s \overset{d}{=} L_{t-s},
\end{equation*}
we conclude that
\begin{align*}
&\EE|y(t)-y(s)|^{\k} =\EE\left|\int_{s}^{t}f(r,y(r))dr + \int_{s}^{t}g(r,y(r))dB(r) + (L(t)-L(s))\right|^{\k}\\
&\leq 3^{\k-1}\left(\EE\left|\int_{s}^{t}f(r,y(r))dr\right|^{\k}+\EE\left|\int_{s}^{t}g(r,y(r))dB(r)\right|^{\k}+\EE|L(t)-L(s)|^{\k}\right)\\
&\leq 3^{\k-1}\left(|t-s|^{\k-1}\int_{s}^{t}\EE|f(r,y(r))|^{\k}dr + c_{\k}\left|\int_{s}^{t}\EE|g(r,y(r))|^2dr\right|^{\frac{\k}{2}}+\EE|L(t)-L(s)|^{\k}\right)\\
&\leq C|t-s|^{\eta},
\end{align*}
where
the constant $C>0$ depends on $\widetilde{H}$, $C_p$ and $\k$,
and
$$\eta=
\begin{cases}
\frac{\k}{2}, & g(\cdot,\cdot)\neq 0,\\
\frac{\k}{\gamma_0}, & g(\cdot,\cdot)\equiv 0.
\end{cases}$$
This completes the proof.
\end{proof}

\subsection{Proof of Theorem \ref{Conver}}\label{Section3.2}
\begin{proof}
It follows from \eqref{sde} and \eqref{Numsde} that
\begin{equation*}
y(t_{i+1})=y(t_i)+\int_{t_i}^{t_{i+1}}f(s,y(s))ds + \int_{t_i}^{t_{i+1}}g(s,y(s))dB(s)+\int_{t_i}^{t_{i+1}}dL(s)
\end{equation*}
and
\begin{equation*}
Y_{i+1}=Y_i+\int_{t_i}^{t_{i+1}}f(t_{i+1},Y_{i+1})ds + \int_{t_i}^{t_{i+1}}g(t_i,Y_i)dB(s)+\int_{t_i}^{t_{i+1}}dL(s),\quad i=0,1,\cdots, \mathbf{N}.
\end{equation*}
So,
\begin{align*}
y(t_{i+1})-Y_{i+1} &= y(t_i)-Y_i+\int_{t_i}^{t_{i+1}}f(s,y(s))-f(t_{i+1},Y_{i+1})ds
\\
&\quad +\int_{t_i}^{t_{i+1}}g(s,y(s))-g(t_{i},Y_{i})dB(s).
\end{align*}
Multiplying  both sides  by the transpose of $y(t_{i+1})-Y_{i+1}$, we get
\begin{align*}
|y(t_{i+1})-Y_{i+1}|^2 &= \int_{t_i}^{t_{i+1}}(y(t_{i+1})-Y_{i+1})^{\mathrm{T}}(f(s,y(s))-f(t_{i+1},Y_{i+1}))ds\\
&\quad + (y(t_{i+1})-Y_{i+1})^{\mathrm{T}}\left(( y(t_i)-Y_i)+\int_{t_i}^{t_{i+1}}g(s,y(s))-g(t_{i},Y_i)dB(s)\right)\\
&=: J_1+J_2.
\end{align*}
Note that
\begin{align*}
J_1&= \int_{t_i}^{t_{i+1}}(y(t_{i+1})-Y_{i+1})^{\mathrm{T}}(f(s,y(s))-f(t_{i+1},Y_{i+1}))ds\\
&= \int_{t_i}^{t_{i+1}}(y(t_{i+1})-Y_{i+1})^{\mathrm{T}}(f(s,y(t_{i+1}))-f(t_{i+1},y(t_{i+1})))ds\\
&\quad + \int_{t_i}^{t_{i+1}}(y(t_{i+1})-Y_{i+1})^{\mathrm{T}}(f(t_{i+1},y(t_{i+1}))-f(t_{i+1},Y_{i+1}))ds\\
&\quad + \int_{t_i}^{t_{i+1}}(y(t_{i+1})-Y_{i+1})^{\mathrm{T}}(f(s,y(s))-f(s,y(t_{i+1})))ds\\
&=: J_{11}+J_{12}+J_{13}.
\end{align*}
We estimate the terms $ J_{11}, J_{12}$ and $J_{13}$ separately. By Assumption \ref{timeHolder}, we have
\begin{align*}
J_{11}&= \int_{t_i}^{t_{i+1}}(y(t_{i+1})-Y_{i+1})^{\mathrm{T}}(f(s,y(t_{i+1}))-f(t_{i+1},y(t_{i+1})))ds\\
&\leq \frac{1}{2}\left(\int_{t_i}^{t_{i+1}}|y(t_{i+1})-Y_{i+1}|^2ds+ \int_{t_i}^{t_{i+1}}|f(s,y(t_{i+1}))-f(t_{i+1},y(t_{i+1}))|^2ds\right)\\
&\leq \frac{1}{2}\int_{t_i}^{t_{i+1}}|y(t_{i+1})-Y_{i+1}|^2ds+\frac{K_1^2}{2}\int_{t_i}^{t_{i+1}}(1+|y(t_{i+1})|^{2\s+2})|t_{i+1}-s|^{2\gamma_1}ds\\
&\leq \frac{1}{2}\int_{t_i}^{t_{i+1}}|y(t_{i+1})-Y_{i+1}|^2ds+\frac{K_1^2}{2}\D t^{1+2\gamma_1}+\frac{K_1^2}{2}\D t^{1+2\gamma_1}|y(t_{i+1})|^{2\s+2}.
\end{align*}
Due to Assumption \ref{sideLip},
\begin{align*}
J_{12}&= \int_{t_i}^{t_{i+1}}(y(t_{i+1})-Y_{i+1})^{\mathrm{T}}(f(t_{i+1},y(t_{i+1}))-f(t_{i+1},Y_{i+1}))ds\\
&\leq K_3\int_{t_i}^{t_{i+1}}|y(t_{i+1})-Y_{i+1}|^2ds.
\end{align*}
It follows from  Assumption \ref{superlinear} that
\begin{align}
\label{J13}
\nonumber
J_{13}&=\int_{t_i}^{t_{i+1}}(y(t_{i+1})-Y_{i+1})^{\mathrm{T}}(f(s,y(s))-f(s,y(t_{i+1})))ds\\ \nonumber
&\leq \frac{1}{2}\left(\int_{t_i}^{t_{i+1}}|y(t_{i+1})-Y_{i+1}|^2ds+|f(s,y(s))-f(s,y(t_{i+1}))|^2ds\right)\\\nonumber
&\leq \frac{1}{2}H\int_{t_i}^{t_{i+1}}(1+|y(s)|^{\s}+|y(t_{i+1})|^{\s})|y(s)-y(t_{i+1})|^2ds\\
&\quad+ \frac{1}{2}\int_{t_i}^{t_{i+1}}|y(t_{i+1})-Y_{i+1}|^2ds.
\end{align}
Applying the H\"older inequality and Lemmas \ref{exactbound} and \ref{continue} (the case of $g(\cdot,\cdot)\neq 0$), we have
\begin{align}\label{ts}
\EE&\left((1+|y(s)|^{\s} +|y(t_{i+1})|^{\s})|y(s)-y(t_{i+1})|^2\right)\notag\\
&\leq 3^{\frac{2}{\s}}\left(\EE(1+|y(s)|^{\s+2}+|y(t_{i+1})|^{\s+2})\right)^{\frac{\s}{\s+2}}\left(\EE|y(s)-y(t_{i+1})|^{\s+2}\right)^{\frac{2}{\s+2}}\notag\\
&\leq 3^{\frac{2}{\s}}\left(1+\EE|y(s)|^{\s+2}+\EE|y(t_{i+1})|^{\s+2}\right)^{\frac{\s}{\s+2}}\left(\EE|y(s)-y(t_{i+1})|^{\s+2}\right)^{\frac{2}{\s+2}}\notag\\
&\leq 2\times3^{\frac{2}{\s}}(1+C_p)^{\frac{\s}{\s+2}}\left(C|t_{i+1}-s|^{\frac{\s+2}{2}}\right)^{\frac{2}{\s+2}}\notag\\
&\leq C|t_{i+1}-s|.
\end{align}
From the above estimates, it yields
\begin{align}\label{J1eatimate}
\nonumber
\EE[J_1]&\leq (1+K_3)\int_{t_i}^{t_{i+1}}\EE|y(t_{i+1})-Y_{i+1}|^2ds+\frac{K_1^2}{2}\D t^{1+2\gamma_1}\\ \nonumber
&\quad+ \frac{K_1^2}{2}\D t^{1+2\gamma_1}\EE|y(t_{i+1})|^{2\s+2}+\frac{HC}{2}|t_{i+1}-s|^2\\ \nonumber
&\leq (1+K_3)\D t\EE|y(t_{i+1})-Y_{i+1}|^2 + \frac{K_1^2}{2}\D t^{1+2\gamma_1} + \frac{K_1^2C_p}{2}\D t^{1+2\gamma_1}+\frac{HC}{2}\D t^2\\
&\leq (1+K_3)\D t\EE|y(t_{i+1})-Y_{i+1}|^2 + C(\D t^{1+2\gamma_1}+\D t^2).
\end{align}
The term $J_2$ can be estimated as
\begin{align*}
J_2&= (y(t_{i+1})-Y_{i+1})^{\mathrm{T}}\left(( y(t_i)-Y_i)+\int_{t_i}^{t_{i+1}}g(s,y(s))-g(t_{i},Y_i)dB(s)\right)\\
&\leq \frac{1}{2}|y(t_{i+1})-Y_{i+1}|^2+\frac{1}{2}\left|( y(t_i)-Y_i)+\int_{t_i}^{t_{i+1}}g(s,y(s))-g(t_{i},Y_i)dB(s)\right|^2\\
&=: \frac{1}{2}|y(t_{i+1})-Y_{i+1}|^2+J_{21}.
\end{align*}
It follows from the isometric property
\begin{equation*}
\EE[J_{21}]\leq \EE|y(t_i)-Y_i|^2 + \int_{t_i}^{t_{i+1}}\EE|g(s,y(s))-g(t_{i},Y_i)|^2ds.
\end{equation*}
By Assumptions \ref{timeHolder} and \ref{lineargro}, we can easily get
\begin{align*}
 |g&(s, y(s))-g(t_{i},Y_i)|^2\\
 &= |g(s,y(t_i))-g(t_i,y(t_i))+g(t_i,y(t_i))-g(t_i,Y_i)+g(s,y(s))-g(s,y(t_i))|^2\\
&\leq 3|g(s,y(t_i))-g(t_i,y(t_i))|^2+3|g(t_i,y(t_i))-g(t_i,Y_i)|^2+3|g(s,y(s))-g(s,y(t_i))|^2\\
&\leq 6K_2^2(1+|y(t_i)|^{2\s+2})|s-t_i|^{2\gamma_2}+3K_4|y(t_i)-Y_i|^2+3K_4|y(s)-y(t_i)|^2\\
&\leq 6K_2^2\D t^{2\gamma_2}(1+|y(t_i)|^{2\s+2})+3K_4|y(t_i)-Y_i|^2+3K_4|y(s)-y(t_i)|^2.
\end{align*}
So,
\begin{align*}
\int_{t_i}^{t_{i+1}}&\EE|g(s,y(s))-g(t_{i},Y_i)|^2ds\\& \leq 6K_2^2\D t^{2\gamma_2}\int_{t_i}^{t_{i+1}}(1+\EE|y(t_i)|^{2\s+2})ds+3K_4\int_{t_i}^{t_{i+1}}\EE|y(t_i)-Y_i|^2ds\\&\quad
+3K_4\int_{t_i}^{t_{i+1}}\EE|y(s)-y(t_i)|^2ds\\
&\leq 6K_2^2(1+C_p)\D t^{1+2\gamma_2}+3K_4\D t\EE|y(t_i)-Y_i|^2+3K_4C\D t^{2}\\
&\leq 3K_4\D t\EE|y(t_i)-Y_i|^2+C(\D t^{1+2\gamma_2}+\D t^{2}).
\end{align*}
Therefore,
\begin{equation*}
\EE[J_{21}]\leq (1+3K_4\D t)\EE|y(t_i)-Y_i|^2+C(\D t^{1+2\gamma_2}+\D t^{2}).
\end{equation*}
We get the estimate
\begin{equation}\label{J2eatimate}
\EE[J_2]\leq \frac{1}{2}\EE|y(t_{i+1})-Y_{i+1}|^2+(1+3K_4\D t)\EE|y(t_i)-Y_i|^2+C(\D t^{1+2\gamma_2}+\D t^{2}).
\end{equation}
Combination of \eqref{J1eatimate} and \eqref{J2eatimate} gives
\begin{align*}
\EE|y(t_{i+1})-Y_{i+1}|^2&\leq (K_3\D t+\D  t+\frac{1}{2})\EE|y(t_{i+1})-Y_{i+1}|^2+C(\D t^{1+2\gamma_1}+\D t^2)\\
&\quad + (1+3K_4\D t)\EE|y(t_i)-Y_i|^2+C(\D t^{1+2\gamma_2}+\D t^{2})\\
&\leq \frac{1+3K_4\D t}{\frac{1}{2}-K_3\D t-\D t}\EE|y(t_i)-Y_i|^2+ C(\D t^{1+2\gamma_1}+\D t^{1+2\gamma_2}+\D t^{2}).
\end{align*}
Summing up the above, we get
\begin{align*}
\sum_{r=1}^{i}\EE|y(t_r)-Y_r|^2\leq \frac{2(1+3K_4\D t)}{1-2K_3\D t-2\D t}\sum_{r=1}^{i-1}\EE|y(t_r)-Y_r|^2+ iC\left(\D t^{1+2\gamma_1}+\D t^{1+2\gamma_2}+\D t^{2}\right).
\end{align*}
Because of $i\D t=t_i\leq e^{t_i}$, we can obtain
\begin{align*}
\EE|y(t_i)-Y_i|^2&\leq \frac{1+2\D t(3K_4+K_3+1)}{1-2K_3\D t-2\D t}\sum_{r=1}^{i-1}\EE|y(t_r)-Y_r|^2\\
&\quad+ Ce^{t_i}\left(\D t^{2\gamma_1}+\D t^{2\gamma_2}+\D t\right),
\end{align*}
by the discrete Gronwall inequality, we get
\begin{equation}\label{yXi}
\EE|y(t_i)-Y_i|^2\leq C\left(\D t^{2\gamma_1}+\D t^{2\gamma_2}+\D t\right)e^{Ct_i}.
\end{equation}
\par
For any $t\in[i\D t, (i+1)\D t)$, it follows from \eqref{consol}, \eqref{yts} and \eqref{yXi} that
\begin{align*}
\EE|y(t)-Y(t)|^2&\leq 2\EE|y(t)-y(t_i)|^2+2\EE|y(t_i)-Y_i|^2\\
&\leq 2\D t +2C\left(\D t^{2\gamma_1}+\D t^{2\gamma_2}+\D t^{\frac{1}{2}}\right)e^{Ct_i}\\
&\leq C\left(\D t^{2\gamma_1}+\D t^{2\gamma_2}+\D t\right).
\end{align*}
This completes the proof.
\end{proof}

\section{Numerical invariant measure}\label{Main2}
In this section, we discuss the stationary distribution of the numerical solution generated by the semi-implicit EM method. In order to simplify the analysis, we consider the following autonomous SDE
\begin{equation}
\label{SDE3}
x(t)=x(0)+\int_{0}^{t}f(x(s))ds+\int_{0}^{t}g(x(s))dB(s)+ L(t), \quad t>0,
\end{equation}
where the L\'evy process $L(t)$ is defined as that in Section \ref{Pre} with an additional requirement that $\mathbb{E}\left(L(t) \right) = 0$. The semi-implicit EM method reduces to
\begin{equation}
\label{autonomonsNum}
X_{i+1}=X_i + f(X_{i+1})\D t + g(X_i)\D B_{i+1}+\D L_{i+1},\quad i=0,1,2,\ldots.
\end{equation}
\par
We need more notations.
Let $\mathcal{P}(\RR^d)$ denote the family of all probability measures on $\RR^d$. For any $k\in(0,1]$, define a metric $d_k(u,v)$ on $\RR^d$ as
\begin{equation*}
d_k(u,v)=|u-v|^k, \quad u,v\in\RR^d.
\end{equation*}
Define the corresponding Wasserstein distance between $\omega\in\mathcal{P}(\RR^d)$ and $\omega'\in\mathcal{P}(\RR^d)$ by
\begin{equation*}
W_k(\omega,\omega')=\inf\EE(d_k(u,v)),
\end{equation*}
where the infimum is taken over all pairs of random variables $u$ and $v$ on $\RR^d$  with respect to the laws $\omega$ and $\omega'$.\par
Let $\overline{\mathbb{P}}_t(\cdot,\cdot)$ be the transition probability kernel of the underlying solution $x(t)$, with the notation $\delta_x\overline{\mathbb{P}}_t$ emphasizing the initial value $x$. The probability measure $\pi(\cdot)\in\mathcal{P}(\RR^d)$ is called an invariant measure of $x(t)$ if
\begin{equation*}
\pi(B)=\int_{\RR^d}\overline{\mathbb{P}}_t(x,B)\pi(dx)
\end{equation*}
holds for any $t\geq 0$ and Borel set $B\in \mathcal{B}(\RR^d)$.

Let $\mathbb{P}_i(\cdot,\cdot)$ be the transition probability kernel of the numerical solution  $\{X_i\}_{i\geq 0}$ with the notation $\delta_x\mathbb{P}_i$ emphasizing the initial value $x$. The probability measure $\Pi_{\D t}(\cdot)\in\mathcal{P}(\RR^d)$  is called an invariant measure of the numerical solution $\{X_i\}_{i\geq 0}$ if
 \begin{equation*}
\Pi_{\D t}(B)=\int_{\RR^d}\mathbb{P}_i(x,B)\Pi_{\D t}(dx),~~\forall i=0,1,2,...,
\end{equation*}
 for any $B\in \mathcal{B}(\RR^d)$.

\begin{lemma}\label{Xiyizhi}
Suppose that \eqref{sideLipx} and \eqref{lineg} hold. Then  for any $\D t\in(0,1)$, the numerical solution is uniformly bounded by
\begin{equation*}
\EE|X_i|^2\leq Q_1^i\EE|X_0|^2+\frac{Q_2(1-Q_1^i)}{1-Q_1},\quad i=1,2,\ldots, \mathbf{N},
\end{equation*}
where
\begin{equation*}
Q_1=\frac{1+M_2\D t}{1-2M_1\D t}<1\ \  \textrm{and} \quad Q_2=\frac{(2m_1+m_2+1)\D t}{1-2M_1\D t}.
\end{equation*}
\end{lemma}
\begin{proof}
Multiplying both sides of \eqref{autonomonsNum} with the transpose of $X_{i+1}$ yields
\begin{equation*}
|X_{i+1}|^2=X^{\mathrm{T}}_{i+1}f(X_{i+1})\D t +X_{i+1}^{\mathrm{T}}\left(X_i+g(X_i)\D B_{i+1}+\D L_{i+1}\right).
\end{equation*}
It follows from  \eqref{sideLipx}, \eqref{lineg},  $\D t^{2/\gamma_0}<\D t<1$  and the elementary inequality that
\begin{align*}
\EE|X_{i+1}|^2&\leq \frac{2m_1\D t}{1-2M_1\D t}+\frac{1}{1-2M_1\D t}\EE|X_i|^2+\frac{M_2\D t}{1-2M_1\D t}\EE|X_i|^2\\
&\quad +\frac{m_2\D t}{1-2M_1\D t}+\frac{1}{1-2M_1\D t}\D t^{2/\gamma_0}\\
&\leq \frac{1+M_2\D t}{1-2M_1\D t}\EE|X_i|^2+\frac{(2m_1+m_2+1)\D t}{1-2M_1\D t}\\
&= Q_1\EE|X_i|^2+Q_2.
\end{align*}
Since $M_2+2M_1<0$,
\begin{equation*}
Q_1=\frac{1+M_2\D t}{1-2M_1\D t}<1
\end{equation*}
for any $\D t\in(0,1)$.
Thus
\begin{equation*}
\EE|X_i|^2\leq Q_1^i\EE|X_0|^2+\frac{Q_2(1-Q_1^i)}{1-Q_1}.
\end{equation*}
This completes the proof.
\end{proof}
Let $\{X_i^x\}_{i\geq 0}$ and $\{X_i^y\}_{i\geq 0}$ be the numerical solutions with respect to  two different initial values  $x$ and $y$.
\begin{lemma}\label{Y2}
Given Assumptions \ref{sideLip} and \ref{lineargro}. For any $\D t\in(0,1)$, the numerical solutions satisfy
\begin{equation*}
\lim_{i\rightarrow \infty}\EE|X_i^x-X_i^y|^2=0.
\end{equation*}
\end{lemma}
\begin{proof}
Note that
\begin{equation*}
X_{i+1}^x-X_{i+1}^y=X_i^x-X_i^y + (f(X_{i+1}^x)-f(X_{i+1}^y))\D t + (g(X_i^x)-g(X_i^y))\D B_{i+1}.
\end{equation*}
Multiplying both sides with the transpose of $X_{i+1}^x-X_{i+1}^y$ yields
\begin{align*}
|X_{i+1}^x-X_{i+1}^y|^2&=(X_{i+1}^x-X_{i+1}^y)^{\mathrm{T}}(f(X_{i+1}^x)-f(X_{i+1}^y))\D t \\
&\quad +(X_{i+1}^x-X_{i+1}^y)^{\mathrm{T}}\left(X_i^x-X_i^y + (g(X_i^x)-g(X_i^y))\D B_{i+1}\right).
\end{align*}
By Assumption \ref{sideLip} and $ab\leq (a^2+b^2)/2$, we have
\begin{equation*}
|X_{i+1}^x-X_{i+1}^y|^2 \leq (\frac{1}{2}+K_3\D t)|X_i^x-X_i^y|^2+\frac{1}{2}|X_i^x-X_i^y + (g(X_i^x)-g(X_i^y))\D B_{i+1}|^2.
\end{equation*}
Taking expectations on both sides above and using Assumption \ref{lineargro} result in
\begin{equation*}
\EE|X_{i+1}^x-X_{i+1}^y|^2\leq \frac{1+K_4\D t}{1-2K_3\D t}\EE|X_i^x-X_i^y|^2.
\end{equation*}
Due to  $K_4+2K_3<-1$,
\begin{equation*}
\frac{1+K_4\D t}{1-2K_3\D t}<1
\end{equation*}
holds for any $\D t\in(0, 1)$. Therefore,
\begin{equation*}
\EE|X_i^x-X_i^y|^2\leq Q_3^i\EE|X_0^x-X_0^y|^2.
\end{equation*}
where $Q_3=\frac{1+K_4\D t}{1-2K_3\D t}$. This completes the proof.
\end{proof}
We now present the existence and uniqueness  of the invariant measure of the numerical solution $\{X_i\}_{i\geq 0}$.
\begin{theorem}\label{Invariant}
Suppose that Assumptions \ref{sideLip} and \ref{lineargro} hold. Then for any fixed $\D t\in(0,1)$, the numerical solution $\{X_i\}_{i\geq 0}$ has a unique invariant measure $\Pi_{\D t}$.
\end{theorem}
\begin{proof}
For each integer $n\geq 1$ and any Borel set $B\subset \RR^d$, define the measure
\begin{equation*}
\omega_n(B)={\frac{1}{n}\sum_{i=1}^{n}\mathbb{P}(X_i\in B)}.
\end{equation*}
It follows from Lemma \ref{Xiyizhi} and the Chebyshev inequality that the measure sequence $\{\omega_n(B)\}_{n\geq 1}$ is tight such that there exists a subsequence which converges to an invariant measure \cite{BaoShaoYuan2016}. This proves the existence of the invariant measure of the numerical solution. In the following, we show  the uniqueness of the invariant measure.
Let $\Pi_{\D t}^{x}$ and $\Pi_{\D t}^{y}$ be invariant measures of $\{X_i^x\}_{i\geq 0}$ and $\{X_i^y\}_{i\geq 0}$, respectively. Then
\begin{align*}
W_k(\Pi_{\D t}^x,\Pi_{\D t}^y)&= W_k(\Pi_{\D t}^x\mathbb{P}_i,\Pi_{\D t}^y\mathbb{P}_i)\\
&\leq \int_{\RR^d}\int_{\RR^d}\Pi_{\D t}^x(dx)\Pi_{\D t}^y(dy)W_k(\delta_x\mathbb{P}_i,\delta_y\mathbb{P}_i).
\end{align*}
From Lemma \ref{Y2}, we get
\begin{equation*}
W_k(\delta_x\mathbb{P}_i,\delta_y\mathbb{P}_i)\leq (Q_3^i\EE|x-y|^2)^{\frac{k}{2}} \rightarrow 0, \quad \text{as} \quad i\rightarrow \infty.
\end{equation*}
Then, we have
\begin{equation*}
\lim_{i\rightarrow \infty}W_k(\Pi_{\D t}^x,\Pi_{\D t}^y)=0,
\end{equation*}
which completes the proof.
\end{proof}
The following theorem states the numerical invariant measure $\Pi_{\D t}$ converges to the underlying one $\pi$ in the Wassertein distance.
\begin{theorem}\label{conv}
Suppose that Assumptions \ref{sideLip} and \ref{lineargro} hold. Then 
\begin{equation*}
\lim_{\D t\rightarrow 0+}W_k(\pi,\Pi_{\D t})=0.
\end{equation*}
\end{theorem}
\begin{proof}
For any $k\in(0,1]$,
\begin{equation*}
W_k(\delta_x\overline{\mathbb{P}}_{i\D t},\pi)\leq \int_{\RR^d}\pi(dy)W_k(\delta_x\overline{\mathbb{P}}_{i\D t},\delta_y\overline{\mathbb{P}}_{i\D t})
\end{equation*}
and
\begin{equation*}
W_k(\delta_x\mathbb{P}_{i},\Pi_{\D t})\leq \int_{\RR^d}\Pi_{\D t}(dy)W_k(\delta_x\mathbb{P}_{i},\delta_y\mathbb{P}_{i}).
\end{equation*}
Because of the existence and uniqueness of the invariant measure for the numerical method \eqref{autonomonsNum} and Theorem \ref{Invariant}, for any $\D t\in(0,1)$ and $\e>0$, there exists a $i>0$ sufficiently large such that
\begin{equation*}
W_k(\delta_x\overline{\mathbb{P}}_{i\D t},\pi)\leq \frac{\e}{3} \quad \text{and} \quad W_k(\delta_x\mathbb{P}_{i},\Pi_{\D t})\leq \frac{\e}{3}.
\end{equation*}
Then, for the chosen $i$, by Theorem \ref{Conver}, we get
\begin{equation*}
W_k(\delta_x\overline{\mathbb{P}}_{i\D t},\delta_x\mathbb{P}_{i})\leq \frac{\e}{3}.
\end{equation*}
Therefore,
\begin{equation*}
W_k(\pi, \Pi_{\D t})\leq W_k(\d_x\overline{\mathbb{P}}_{i\D t},\pi)+W_k(\delta_x\mathbb{P}_{i},\Pi_{\D t})+W_k(\d_x\overline{\mathbb{P}}_{i\D t},\delta_x\mathbb{P}_{i})\leq \e,
\end{equation*}
which completes the proof.
\end{proof}

\begin{rmk}
Our conclusions reveal the existence and uniqueness of numerical invariant measure in the sense of k-Wasserstein distance and convergence of the numerical invariant measure to the underlying counterpart for any $k\in(0,1]$.
\end{rmk}

\section{Numerical examples}\label{Numex}
In this section, we conduct some numerical experiments to verify our theoretical results. Examples \ref{ex_conver} and \ref{ex_conver2} are used to illustrate the results of Theorem \ref{Conver} and Corollary \ref{Conver2}, respectively. Examples \ref{ex_stable} and \ref{ex_stable2} is given to demonstrates the results of  Theorems \ref{Invariant} and  \ref{conv}.
\par
It is should be noted that the L\'evy process we used in our numerical experiments is the tempered stable process, whose generating algorithm is borrowed from \cite{ErnestJum2015}.

\begin{expl}\label{ex_conver}
Consider the following SDE
\begin{equation*}
d y(t)=\left([(t-1)(2-t)]^{\frac{1}{5}}y^2(t)-2y^5(t)\right)dt+\left(2[(t-1)(2-t)]^{\frac{2}{5}}y(t)\right)d B(t)+d L(t),
\end{equation*}
where $t\in (0,1]$,  $y(0)=1$.
\end{expl}
For any $x,y\in\RR$,
\begin{align*}
|f(t,x)-f(t,y)|^2&=\left|[(t-1)(2-t)]^{\frac{1}{5}}(x^2-y^2)-2(x^5-y^5)\right|^2\\
&\leq \left|[(t-1)(2-t)]^{\frac{1}{5}}(x+y)-(x^4+y^4)\right|^2|x-y|^2\\
&\leq H(1+|x|^8+|y|^8)|x-y|^2,
\end{align*}
which shows that Assumption \ref{superlinear} is satisfied with $\s=8$.
\par
Using the same approach, for any $t\in(0,1)$, we can prove that Assumptions \ref{Khasminskii}, \ref{sideLip} and \ref{lineargro} are satisfied. Assumption \ref{timeHolder}  also holds with $\gamma_1=1/5$ and $\gamma_2=2/5$. Thus, according to the Theorem \ref{Conver}, we obtain
\begin{equation*}
\EE|y(t)-Y(t)|^2\leq C\left(\D t^{2\gamma_1}+\D t^{2\gamma_2}+\D t\right).
\end{equation*}
Note that the parameter $\gamma_0$ does not affect the convergence order.\par

We conduct $1000$ independent trajectories with different step sizes 
$2^{-9}$, $2^{-10}$, $2^{-11}$, $2^{-12}$ and $2^{-15}$, respectively. Since the true solution of the SDE is difficult to obtain, we take the numerical solution solved with the minimum step size $2^{-15}$ as a reference solution.  We choose different $\gamma_0$ and record the errors versus the step sizes. Figures \ref{fig:1-1} and \ref{fig:1-2}  show that the convergence order is about $0.2$ when  $\gamma_0=1.3, 1.5$.  For $\gamma_0=1.3$, we take $f(t,y)=[(t-1)(2-t)]^{4/5}y^2(t)-2y^5(t)$, and thus $\gamma_1=0.8$. Figure \ref{fig:1-3} shows the convergence order is about $1/2$.
Numerical results confirm a close agreement between theoretical and numerical convergence order.

\begin{figure}[H]
\centering

\subfigure[{\label{fig:1-1}}$\gamma_0=1.3$, $\gamma_1=0.2$, $\gamma_2=0.4$]{
\begin{minipage}[t]{0.5\linewidth}
\centering
\includegraphics[width=2.9in]{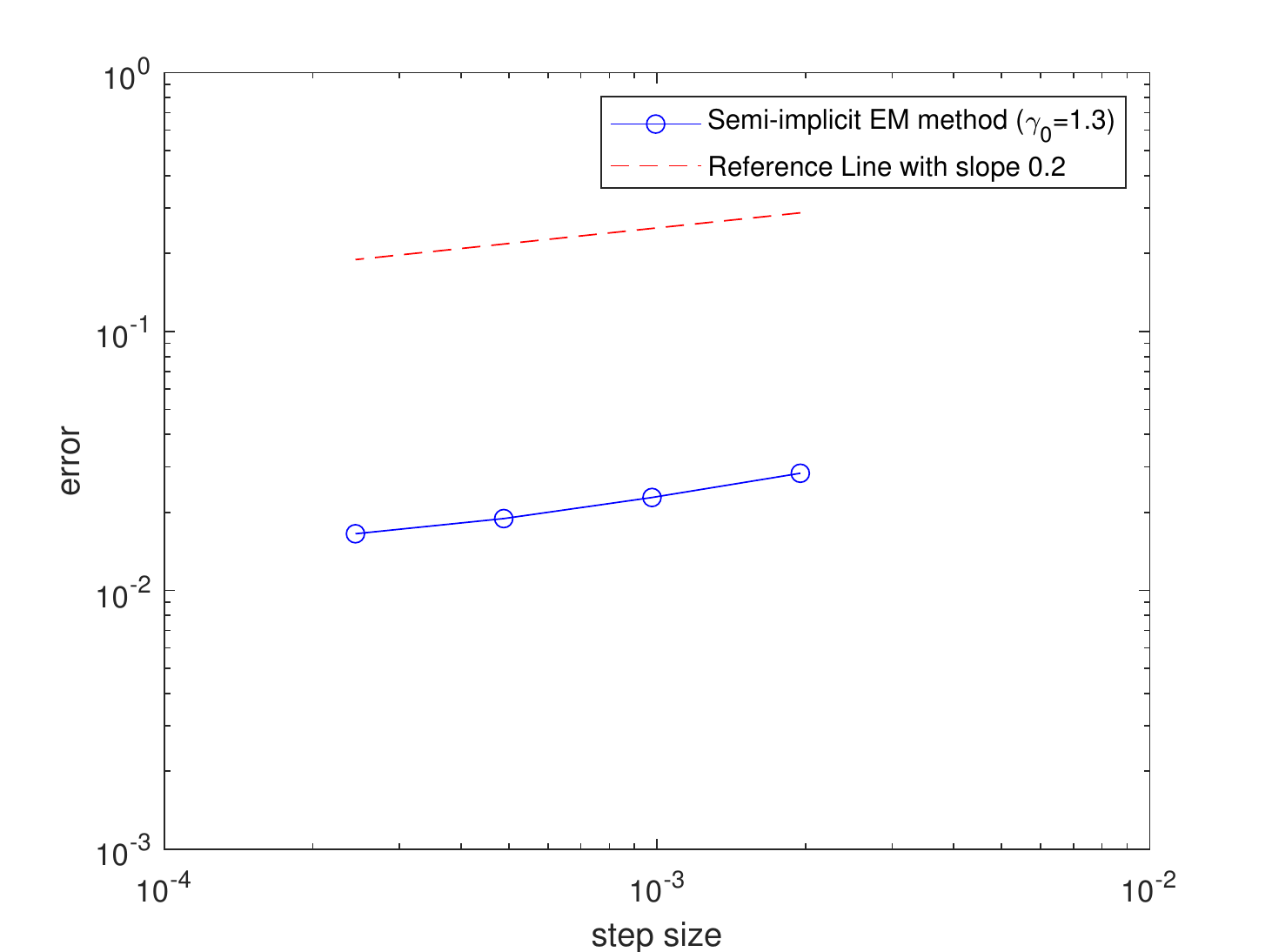}
\end{minipage}%
}%
\subfigure[{\label{fig:1-2}}$\gamma_0=1.5$, $\gamma_1=0.2$, $\gamma_2=0.4$]{
\begin{minipage}[t]{0.5\linewidth}
\centering
\includegraphics[width=2.9in]{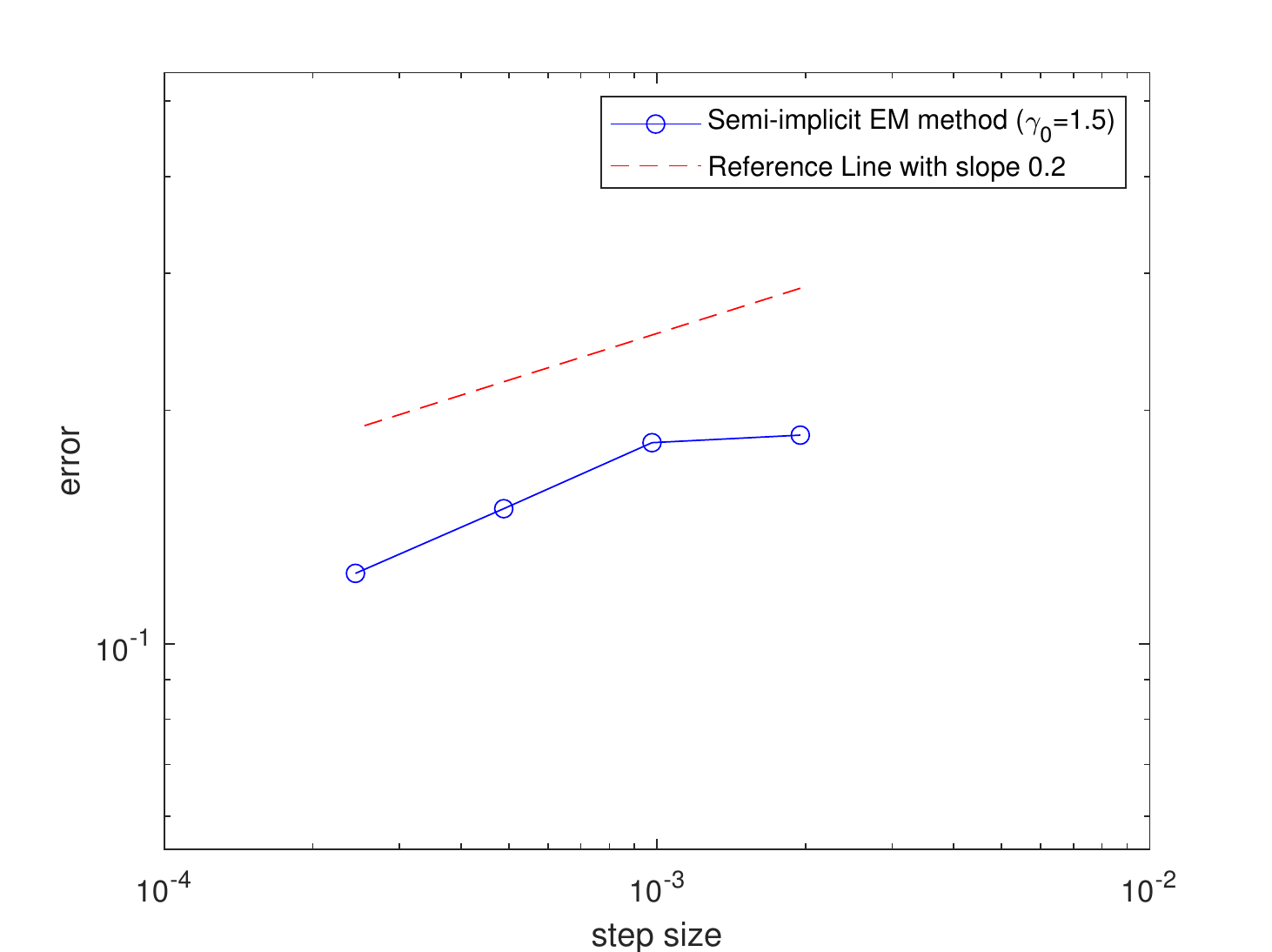}
\end{minipage}%
}%

\subfigure[{\label{fig:1-3}}$\gamma_0=1.3$, $\gamma_1=0.8$, $\gamma_2=0.6$]{
\begin{minipage}[t]{0.5\linewidth}
\centering
\includegraphics[width=2.9in]{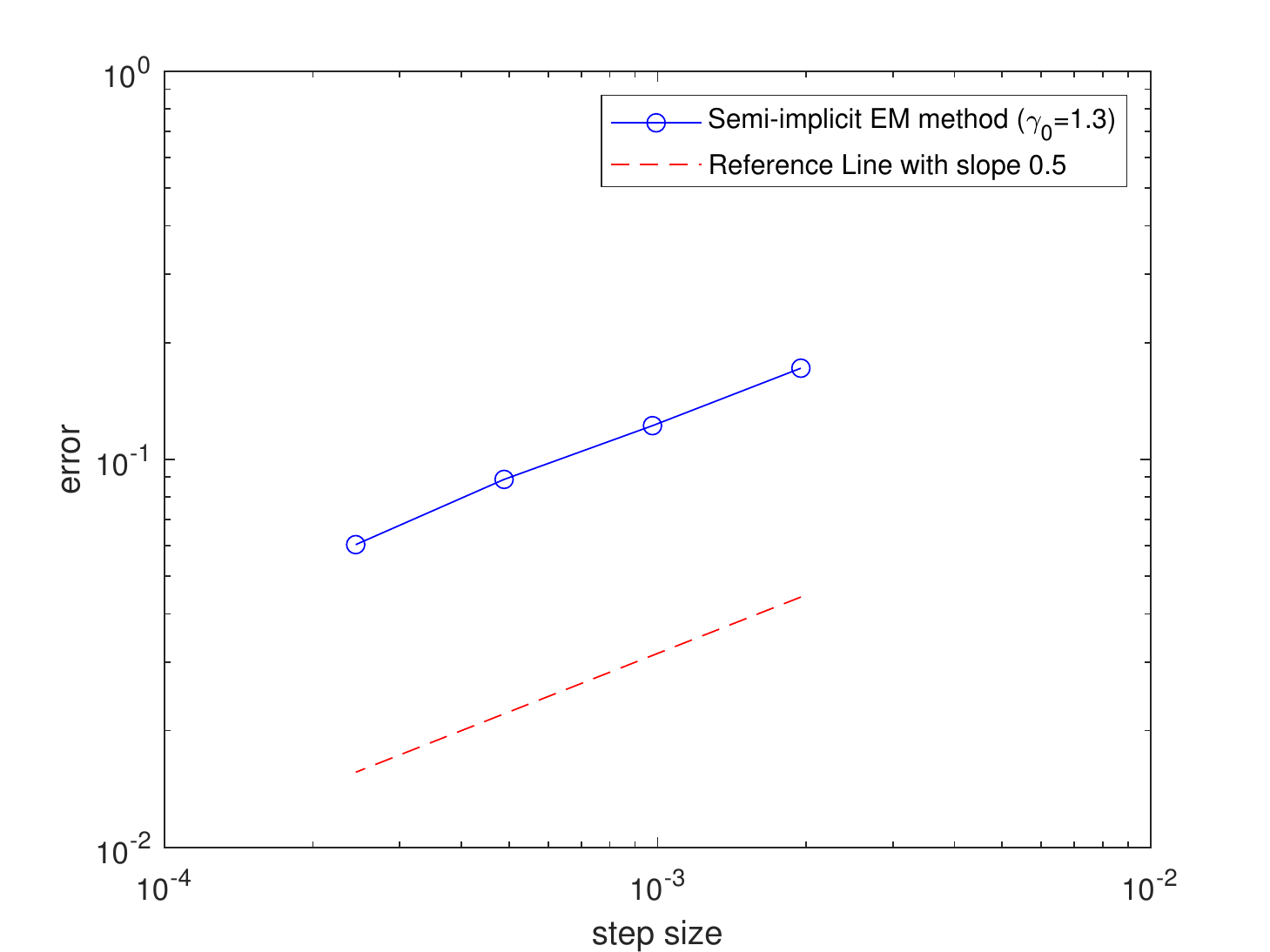}
\end{minipage}
}%

\centering
\caption{Errors versus step sizes}
\end{figure}

\begin{expl}\label{ex_conver2}
Consider the following SDE
\begin{equation*}
d u(t)=\left([(t-1)(2-t)]^{\gamma_1}u^2(t)-2u^5(t)\right)dt+d L(t),\quad t\in[0,1],
\end{equation*}
where $u(0)=1$.
\end{expl}
The Brownian motion noise doesn't appear in this example.
 According to Corollary \ref{Conver2},  the convergence order of the semi-implicit EM method is $\min\{\gamma_1,1/\gamma_0\}$.
For $\gamma_1=0.9$ and $\gamma_0=1.3$, we observe from Figure \ref{fig:2-1} that the convergence order is close to $1/\gamma_0\approx 0.77$, which agrees with Corollary \ref{Conver2}.

\begin{figure}[H]
\centering
  \includegraphics[width=0.7\textwidth]{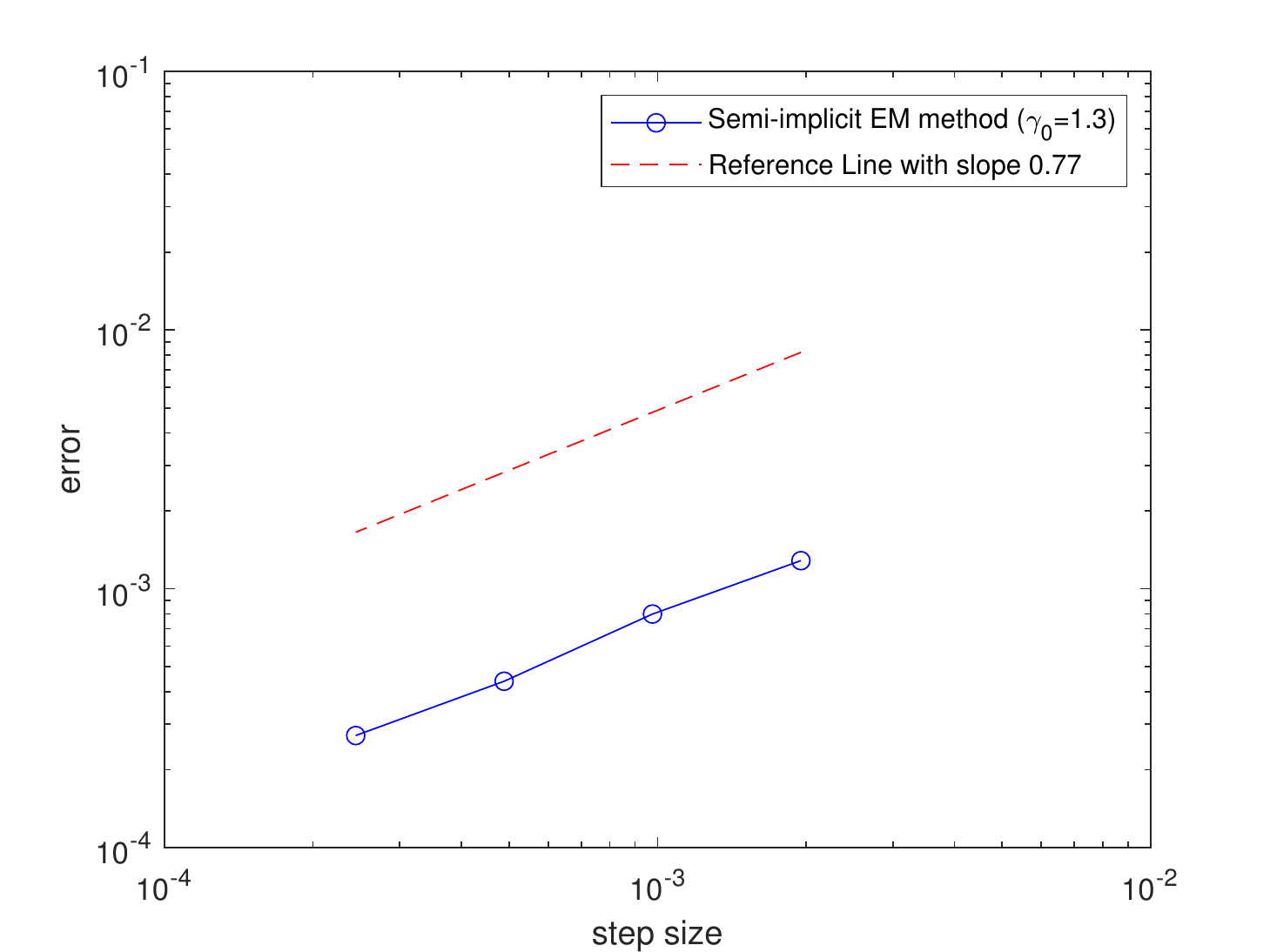}
\caption{Errors versus step sizes}
\label{fig:2-1}       
\end{figure}

\begin{expl}\label{ex_stable}
Consider the following Ornstein-Uhlenbeck (OU) process
\begin{equation*}
dx(t)=-2x(t)dt+2L(t), ~~ x(0)=10.
\end{equation*}
\end{expl}
The Kolmogorov-Smirnov test \cite{Massey1951} helps us measure the difference in distribution between the numerical and true solutions more clearly. 
When $L(t)$ is an $\a$-stable process ($\a\in(0,2)$), the SDE has an invariant measure, which is $\mathbb{S}\left(2\left(\frac{1}{2\a}\right)^{1/\a},0,0\right)$
\cite{ChengHuLong2020}. We choose $\a=1.5$ as the true distribution. We simulate 10000 paths generated by the semi-implicit EM method. The empirical distributions at $t=0.1$, $t=0.3$, $t=0.7$ and $t=2$ are plotted in Figure \ref{fig:3-1}. It can be seen that the empirical distributions get closer to the true distribution as time $t$ increases.  Figure \ref{fig:3-2} shows that the distribution difference between the numerical solution and the true solution decreases as time $t$ increases.

\begin{figure}[H]
\centering
\subfigure[{\label{fig:3-1}} Empirical density functions at different time points]
{\includegraphics[width=0.49\linewidth]{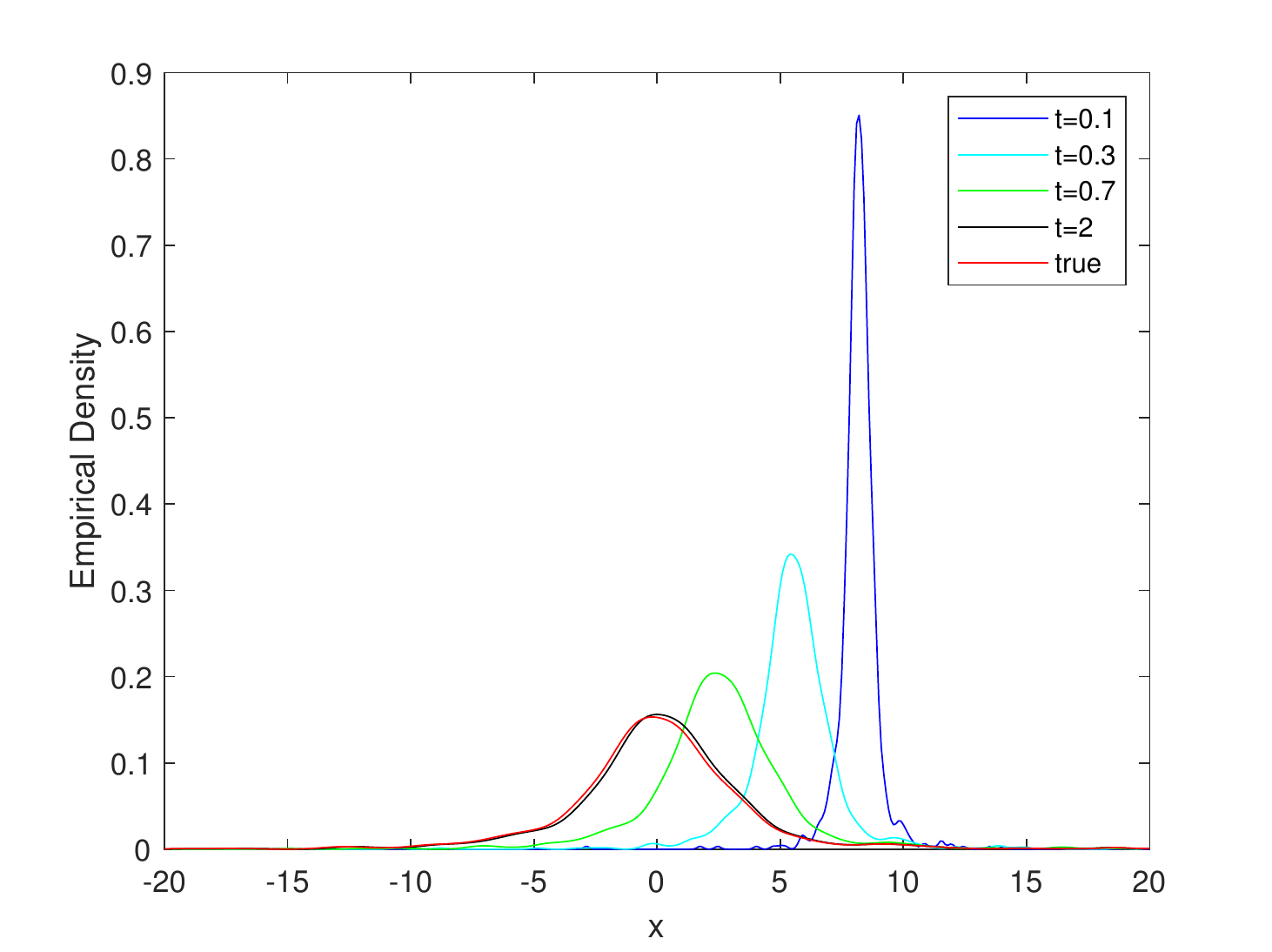} }
\subfigure[{\label{fig:3-2}} Differences between empirical distributions and the true stationary distribution]
{\includegraphics[width=0.49\linewidth]{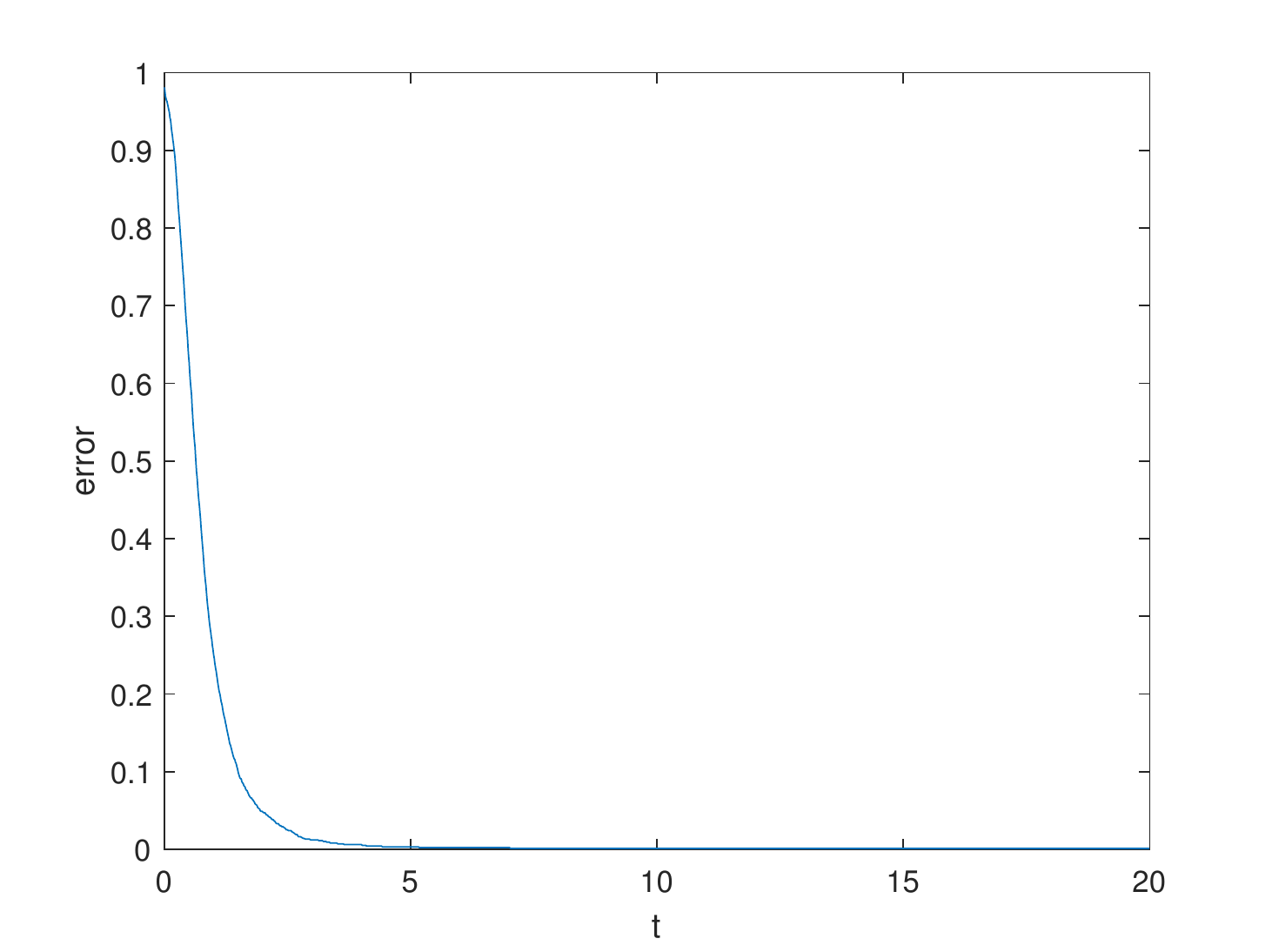}}
\caption{Empirical distributions and long time stability}
\end{figure}

\begin{expl}\label{ex_stable2}
Considering the following SDE
\begin{equation*}
dx(t)=\left(-x^3(t)-5x(t)+5\right)dt+\left(-x(t)+3\right)dB(t)+2dL(t),~~ x(0)=10.
\end{equation*}
\end{expl}

\begin{figure}[H]
\centering
\subfigure[{\label{fig:4-1}} Empirical density functions at different time points]
{\includegraphics[width=0.49\linewidth]{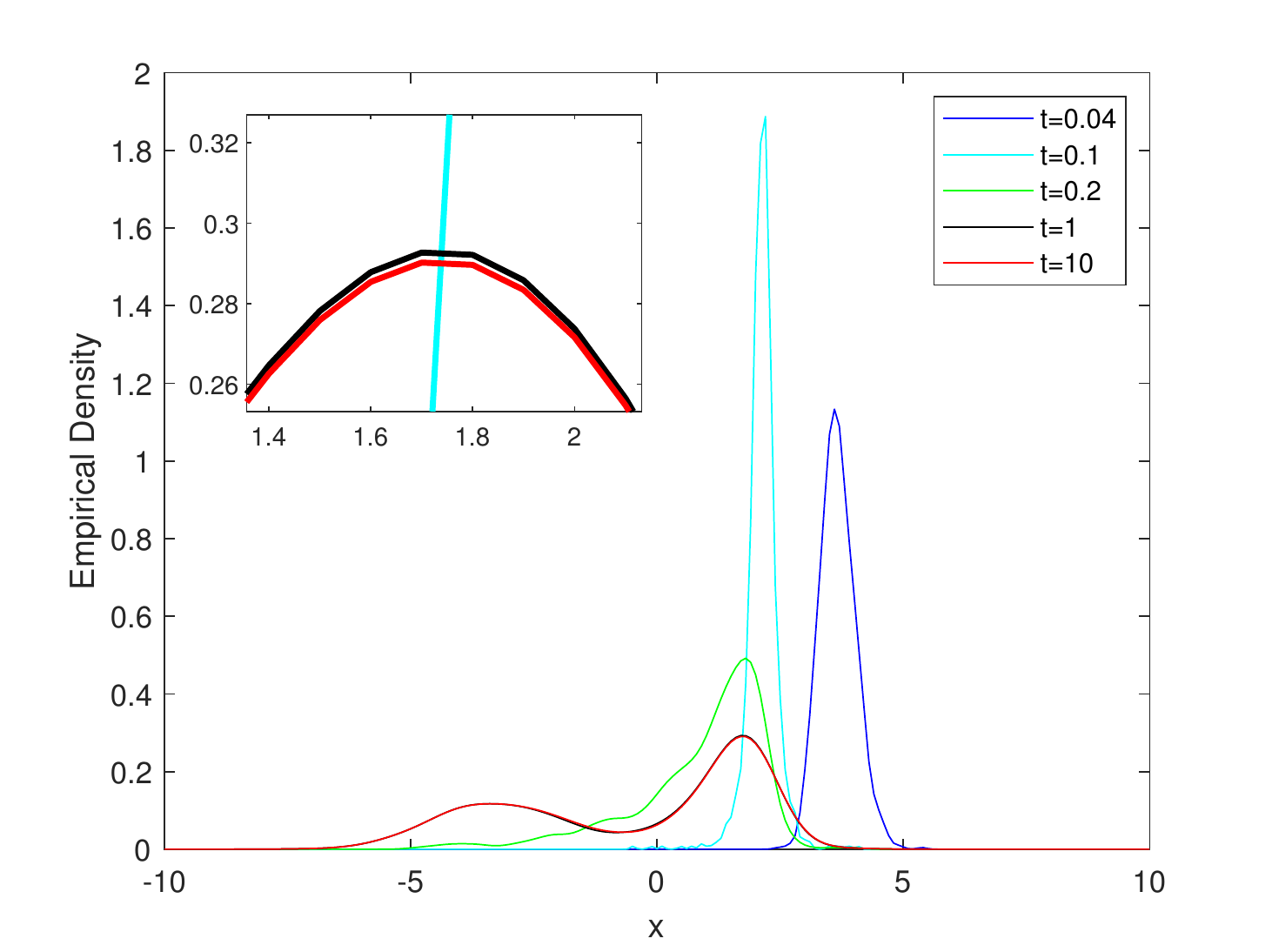} }
\subfigure[{\label{fig:4-2}} Differences between empirical distributions and the true stationary distribution]
{\includegraphics[width=0.49\linewidth]{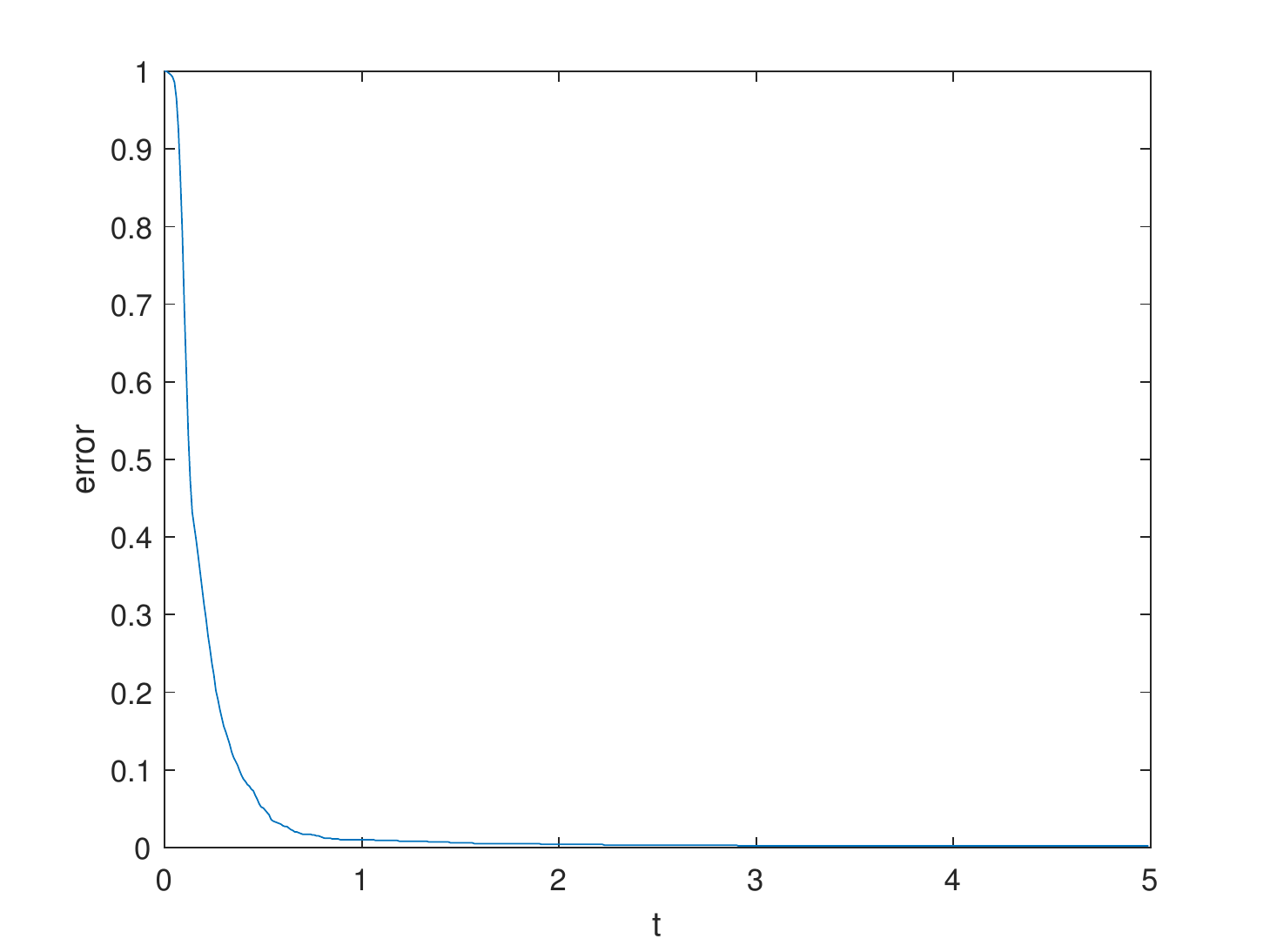}}
\caption{Empirical distributions and long time stability}
\end{figure}

We assume that step size $h=0.01$, $T=10$, and simulate 10000 paths. The empirical distributions at $t=0.04$, $t=0.1$, $t=0.2$, $t=1$ and $t=10$ are plotted in Figure \ref{fig:4-1}. It can be clearly seen that the empirical density functions with $t=0.04$, $t=0.1$ and $t=0.2$ have quite different shapes, but the shapes with $t=2$ and $t=10$ are very close. Figure \ref{fig:4-2} shows that the distribution difference between the numerical solution and the reference solution decreases as time $t$ increases, which supports our theoretical results as well.

\section{Conclusion and future research}\label{Con}
We investigate the finite time strong convergence of the semi-implicit EM method for SDEs with super-linear drift coefficient driven by a class of L\'evy processes. One of the key findings is that the convergence order is related to the parameter of the class of L\'evy process, which has not been observed in literatures. In addition, the semi-implicit EM method is capable of providing a good approximation of the invariant measure of the underlying SDEs.
\par
There are some technical difficulties still to be overcome  to deal with the multiplicative case of the class of L\'evy processes. Furthmore, other stable processes covered by the settings of the class of L\'evy processes in Section \ref{Pre} are worth to try in the future.





\end{document}